\documentclass{amsart}
\usepackage{amsmath}
\usepackage{amssymb}
\usepackage{graphicx}
\input xy
\xyoption{all}
\usepackage{color}
\usepackage{enumerate}
\newtheorem{theorem}{Theorem}[section]

\newtheorem{lemma}[theorem]{Lemma}

\theoremstyle{definition}

\newcommand{\M}{{\rm M}}

\newcommand{\GL}{{\rm GL}}
\newcommand{\SL}{{\rm SL}}

\newcommand{\cor}{{\rm Core}}

\begin{document}
	
	\title[On locally solvable subgroups]{On locally solvable subgroups in division rings}
	
	\author[Hu\`{y}nh Vi\d{\^{e}}t Kh\'{a}nh]{Hu\`{y}nh Vi\d{\^{e}}t Kh\'{a}nh}
	\address{Faculty of Mathematics and Computer Science, VNUHCM - University of Science, 227 Nguyen Van Cu Str., Dist. 5, Ho Chi Minh City, Vietnam.} 
	\email{huynhvietkhanh@gmail.com} 
	
	\maketitle
\begin{abstract} 
	Let $D$ be a division ring with center $F$, and $G$ a subnormal subgroup of $D^*$. We show that if $G$ is a locally solvable group such that a derived subgroup $G^{(i)}$ is algebraic over $F$, then $G$ must be central. Also, if $M$ is non-abelian locally solvable maximal subgroup of $G$ with $M^{(i)}$ algebraic over $F$, then $D$ is a cyclic algebra of prime degree over $F$.
\end{abstract}

\section{Introduction} 

A well-known result of L. K. Hua says that if the multiplicative subgroup $D^*$ of a division ring $D$ is solvable, then $D$ is a field. The subnormal subgroups of $D^*$ have been studied for a long time by many authors. Subnormal subgroups of some special types were considered, such as nilpotent, solvable, and locally nilpotent subgroups. In this direction, Stuth \cite{stuth} asserted that every solvable subnormal subgroup of $D^*$ is central, i.e, it is contained in the center $F$ of $D$. In \cite{huz}, Huzurbazar showed that this result remains true if the word ``solvable" is replaced by ``locally nilpotent''. It is more difficult to handle the case ``locally solvable", and it is unknown whether or not every locally solvable subnormal subgroup in a division ring is central. Relating to this problem, the authors in \cite{hai-thin} showed that
the question has the positive answer in the case when $D$ is algebraic over $F$. To the best of our knowledge there has been
no better results until now. In this note, we show that every locally solvable subnormal subgroup whose $i$-th derived subgroup (for some $i\geq 1$) is algebraic over $F$ is central, which is a slight generalization of those results.

In another direction, maximal subgroups of a subnormal subgroup in division rings were also considered (see e.g. \cite{hai-thin}, \cite{hai-tu}). A remarkable result in \cite{hai-ha} asserted that if $M$ is a non-abelian locally solvable maximal subgroup of $D^*$ such that $M'$ is algebraic over $F$, then $[D:F]<\infty$. This result was generalized by the authors in \cite{hai-tu}, where it is shown that $D$ is even a cyclic algebra of prime degree over $F$ (Theorem 4.2). We show that the result is also true if $D^*$ is replaced by an arbitrary subnormal subgroups and we only need the condition that $M^{(i)}$  is algebraic over $F$ instead of $M'$.

Throughout this note, for a ring $R$ with the identity $1\ne0$, the symbol $R^*$ stands for the group of units of $R$. If $D$ is a division ring with the center $F$ and $S\subseteq D$, 
then $F[S]$ and $F(S)$ denotes respectively the subring and the division subring of $D$ generated by  $F\cup S$.
For a group $G$ and a positive integer $i$, the symbol $G^{(i)}$ is the $i$-th derived subgroup of $G$. If $H$ and $K$ are two subgroups in a group $G$, then $N_K(H)$ denotes the set of all elements $k\in K$ such that $k^{-1}Hk\leq H$, i.e., $N_K(H)=K\cap N_G(H)$. If $A$ is a ring or a group, then $Z(A)$ denotes the center of $A$.

\section{Results}

\begin{lemma}[{\cite[3.2]{wehrfritz_91}}]\label{lemma_2.1}
	Let $R$ be a ring, $J$ a subring of $R$, and $H\leq K$ subgroups of the group of units of $R$ normalizing $J$ such that $R$ is the ring of right quotients of $J[H]\leq R$ and $J[K]$ is a crossed product of $J[B]$ by $K/B$ for some normal subgroup $B$ of $K$. Then $K=HB$.
\end{lemma}

\begin{lemma}[{\cite[Corollary 24]{wehrfritz_89}}]\label{lemma_2.2}
	Let $A$ be a one-sided Artinian ring. Suppose that $S$ is a right Goldie subring of $A$ and $G$ a locally solvable subgroup of the group of units of $A$ normalizing $S$. Set $R=S[G]\leq A$ and assume that $R$ is prime. Then $R$ is right Goldie.
\end{lemma}

\begin{lemma}\label{lemma_2.4}
Let $D$ be a division ring with center $F$, and $G$
a subnormal subgroup of $D^*$. If $G$ is abelian-by-locally finite, then $G\subseteq F$. 
\end{lemma}

\begin{proof} Let $A$ be an abelian normal subgroup of $G$ such that $G/A$ locally finite. Then $A$ is an abelian subnormal subgroup of $D^*$, from which it follows by \cite[Theorem 2]{stuth} that $A \subseteq F$. Therefore $A\subseteq G\cap F^*\subseteq Z(G)$, from which it follows that $G/Z(G)$ is locally finite. Consequently, we conclude that $G'$ is locally finite. Now, $G'$ is a torsion subnormal subgroup of $D^*$, and thus $G'\subseteq F$ by \cite[Therem 8]{her}. This implies that $G$ is solvable, and hence it is contained in $F$ by \cite[Theorem 2]{stuth}.
\end{proof}

\begin{lemma}\label{lemma_2.5}
	Let $D$ be a division ring with center $F$, and $G$ a subgroup of $D^*$. If $G/G\cap F^*$ is a locally finite group, then $F(G)$ is locally finite dimensional (as a vector space) over $F$.
\end{lemma}

\begin{proof}
	For any finite subset $\{x_1,x_2,\dots,x_k\}\subseteq F[G]$, we may write
		$$x_i=f_{i_1}g_{i_1}+f_{i_2}g_{i_2}+\cdots+f_{i_t}g_{i_t},$$ 
	where $f_{i_j}\in F$ and $g_{i_j}\in G$. Let $A=\left\langle g_{i_j}\right\rangle$ the subgroup of $G$ generated by all $g_{i_j}$. By hypothesis, the group $AF^*/F^*$ is finite. Let $\{y_1,y_2,\dots,y_n\}$ be a transversal of $F^*$ in $AF^*$ and let 
		$$R=Fy_1+Fy_2+\cdots+Fy_n.$$
	It is clear that $R$ is a division ring containing the set $\{x_1,x_2,\dots,x_k\}$ and is finite dimensional over $F$.
\end{proof}

\begin{theorem}\label{theorem_main1}
	Let $D$ be a division ring with center $F$, and $G$ a locally solvable subnormal subgroup of $D^*$. If $G^{(i)}$ is algebraic over $F$ for some $i$, then $G\subseteq F$.
\end{theorem}

\begin{proof}
	First, we prove that $F[G^{(i)}]$ is a division ring. Since $F(G^{(i)})$ is normalized by $G$, it follows by Stuth's Theorem (\cite[Theorem 1]{stuth}) that either $F(G^{(i)})\subseteq F$ or $F(G^{(i)})= D$. If the first case occurs, then $F[G^{(i)}]= F$; we are done. Assume that $F(G^{(i)})= D$. Let $T=\tau(G^{(i)})$ be the unique maximal periodic normal subgroup of $G^{(i)}$. By \cite[Theorem 8]{her}, we conclude that $T\subseteq F$. It follows by  \cite[Theorem 1.1(c)]{wehrfritz_91} that $F[G^{(i)}]$ is 
	a crossed product over an abelian characteristic subgroup $A$ of  $G^{(i)}$. Let $B$ be a maximal abelian normal subgroup of $G^{(i)}$ containing $A$. We shall show that $G^{(i)}/B$ is a simple group. For,  let $C$ be a normal subgroup of $G^{(i)}$ properly containing $B$. It is clearly that $C$ is a non-abelian subnormal subgroup of $D^*$, hence $F(C)=D$ by Stuth's Theorem. Moreover, Lemma \ref{lemma_2.2} says that $F[C]$ is an Ore domain whose skew field of fractions coincided with $F(C)=D$. Thus, we may apply Lemma \ref{lemma_2.1} to conclude that  $G^{(i)}=CA=C$; recall that $A\subseteq C$. It follows that $G^{(i)}/B$ is simple, as claimed. Since $G^{(i)}/B$  is locally solvable and simple group, it is finite (see \cite[12.5.2, p. 367]{robinson}). Setting $K=F[B]$, then the algebraicity of $B$ implies that $K=F(B)$ is a subfield of $D$ contained in $F[G^{(i)}]$. Because $G^{(i)}/B$ is finite, we conclude that $[F[G^{(i)}]: K]_r<\infty$. Since $[F[G^{(i)}]$ is a domain which is finite dimensional over the subfield $K$, we conclude that $F[G^{(i)}]$ is a division ring.\\
	
	Next, we claim that $G^{(i)}\subseteq F$. Indeed, by what we have proved, it follows  that $F[G^{(i)}]$ is a division ring, which is clearly normalized  by $G$. By Stuth's Theorem (\cite[Theorem 1]{stuth}), either $F[G^{(i)}]\subseteq F$ or $F[G^{(i)}]=D$. If the first case occurs, then we are done. Now suppose that $F[G^{(i)}]=D$. It follows from \cite[Theorem 1.1]{wehrfritz87} that $G^{(i)}$ is abelian-by-locally finite. Since $G$ is a subnormal subgroup of $D^*$, so is $G^{(i)}$. By Lemma \ref{lemma_2.4}, we have $G^{(i)}\subseteq F$. Therefore, in any case $G^{(i)}\subseteq F$, as claimed. In other words, we have $G$ is solvable, and the result follows from \cite[Theorem 2]{stuth}.
\end{proof}

\begin{lemma}\label{lemma_2.3}
	Let $D$ be a division ring with center $F$, and $G$ a subnormal subgroup of $D^*$. Assume that $M$ is a non-abelian locally solvable maximal subgroup of $G$ such that $M^{(i)}$ is algebraic over $F$ for some $i$. Then $F[M^{(i)}]$ is a division ring.
\end{lemma}

\begin{proof}
	Since $F(M^{(i)})$ is normalized by $M$, it follows that $M\subseteq N_G(F(M^{(i)})^*)\subseteq G$. Since $M$ is maximal in $G$, either $M=N_G(F(M^{(i)})^*)$ or $N_G(F(M^{(i)})^*)=G$.  If the first case occurs, then $M^{(i)} \unlhd F(M^{(i)})^*\cap G$ is subnormal in $F(M^{(i)})^*$ contained in $M$. By Theorem \ref{theorem_main1}, we conclude that $M^{(i)}$ is abelian. The algebraicity of $M^{(i)}$ implies that $F[M^{(i)}]=F(M^{(i)}) $ is a field; we are done. If the second case occurs, then  $F(M^{(i)})^*=D$ by Stuth's Theorem. Let $T=\tau(M^{(i)})$ be the unique maximal periodic normal subgroup of $M^{(i)}$. The local solvability of $T$ implies that it is actually a locally finite group (\cite[Lemma 2.2]{khanh-hai}). Since $T$ is a characteristic subgroup of $M^{(i)}$, it is normal in $M$. It follows that $M\subseteq N_G(F(T)^*)\subseteq G$, which yields either $G=N_G(F(T)^*)$ or $M=N_G(F(T)^*)$. The former case implies $F(T)=D$, hence $D$ is a locally finite division ring by Lemma \ref{lemma_2.5}. Since $M$ is locally solvable, it contains no non-cyclic free subgroups, and thus $[D:F]<\infty$ by \cite[Theorem 3.1]{hai-khanh}. If the latter case occurs, then $T\subseteq F(T)^*\cap G$ is subnormal in $F(T)^*$. In view of \cite[Theorem 8]{her}, we conclude that $T$ is contained in the center of $F(T)$, which means $T$ is abelian. There are two possible cases.
	
	\bigskip 
	
	Case 1. $T\not\subseteq F$.
	
	\bigskip
	
	Take $x\in T\backslash F$. Since $x$ is algebraic over $F$, the elements of the set $x^M=\{m^{-1}xm|m\in M\}\subseteq F(T)$ have the same minimal polynomial over $F$; recall that $F(T)$ is a field. This implies that $|x^M|<\infty$, which says that $[M:C_M(x)]<\infty$. If we set $H=\cor_M(C_M(x))$, then $H$ is a normal subgroup of finite index in $M$. The normality of $H$ in $M$ implies that $M\subseteq N_G(F(H)^*) \subseteq G$. Therefore, we have either $N_G(F(H)^*)=G$ or $N_G(F(H)^*)=M$. The first case implies that $F(H)=D$, which means $x\in F$, a contradiction. We may therefore assume that $N_G(F(H))=M$, from which it follows that $ H\subseteq G\cap F(H)^*$ is a subnormal subgroup of $F(H)^*$ contained in $M$. By Theorem \ref{theorem_main1}, we have $H$ is abelian. If we set $K=F(H)$, then the finiteness of $M/H$ implies that $D=F(M)$ is finite dimensional over $K$. This fact yields $[D:F]<\infty$, hence $F[M^{(i)}]=F(M^{(i)})$ is a division ring.
	
	\bigskip 
	
	Case 2. $T\subseteq F$. 
	
	\bigskip
	
	It follows by  \cite[Theorem 1.1(c)]{wehrfritz_91} that $F[M^{(i)}]$ is 
	a crossed product over an abelian characteristic subgroup $A$ of  $M^{(i)}$. Let $B$ be a maximal abelian normal subgroup of $M^{(i)}$ containing $A$. We shall show that $M^{(i)}/B$ is a simple group. For,  let $C$ be a normal subgroup of $M^{(i)}$ properly containing $B$. It is clearly that $C$ is non-abelian. Since $C$ is normal in $M$, we have $M\subseteq N_G(F(C)^*)\subseteq G$, hence either $M= N_G(F(C)^*)$ or $G= N_G(F(C)^*)$. The former case implies that $C$ is abelian, a contradiction. Thus we have $G= N_G(F(C)^*)$, from which it follows that $F(C)=D$. Moreover, Lemma \ref{lemma_2.2} says that $F[C]$ is an Ore domain whose skew field of fractions coincided with $F(C)=D$. Thus, we may apply Lemma \ref{lemma_2.1} to conclude that  $M^{(i)}=CA=C$; recall that $A\subseteq C$. This fact shows that $M^{(i)}/B$ is simple, as claimed. Since $M^{(i)}/B$  is locally solvable and simple group, it is finite. Setting $K=F[B]$, then the algebraicity of $B$ implies that $K=F(B)$ is a subfield of $D$ contained in $F[M^{(i)}]$. Because $M^{(i)}/B$ is finite, we conclude that $[F[M^{(i)}]: K]_r<\infty$. Now, $[F[M^{(i)}]$ is a domain which is finite dimensional over the subfield $K$, hence $F[M^{(i)}]$ is a division ring.
\end{proof}

Recall that for a group $G$, the set of all elements with finite conjugate classes forms a subgroup of $G$. Such subgroup is called the $FC$-center of $G$. 
\begin{theorem}
	Let $D$ be a division ring with center $F$, and $G$ a subnormal subgroup of $D^*$. Assume that $M$ is a non-abelian locally solvable maximal subgroup of $G$ such that $M^{(i)}$ is algebraic over $F$ for some $i$. Then, the following hold:
	\begin{enumerate}[(i)]
		\item There exists a maximal subfield $K$ of $D$ such that $K/F$ is a finite Galois extension with $\mathrm{Gal}(K/F)\cong M/K^*\cap G\cong \mathbb{Z}_p$ for  some prime $p$, and $[D:F]=p^2$. 
		\item The subgroup $K^*\cap G$ is the $FC$-center. Also, $K^*\cap G$ is  the Fitting subgroup of $M$. Furthermore, for any $x\in M\setminus K$, we have $x^p\in F$ and $D=F[M]=\bigoplus_{i=1}^pKx^i$.
	\end{enumerate}
\end{theorem}
\begin{proof}
	First, we show that $[D:F]<\infty$. If we set $R=F[M^{(i)}]$, then $M\subseteq N_G(R^*) \subseteq G$. By the maximality of $M$ in $G$, it follows that either $N_G(R^*)=M$ or $N_G(R^*)=G$. We need to consider two possible cases:
	
	\bigskip	
	
	\textit{Case 1:} $N_G(R^*)=M$.
	
	\bigskip
	
	In this case, we see that $R^*\cap G\subseteq M$, from which we conclude that $M^{(i)} \unlhd R^*\cap G$. This implies that $M^{(i)}$ is a subnormal subgroup of $R^*$. Moreover, by Lemma \ref{lemma_2.3}, $R$ is a division ring. According  Lemma \ref{lemma_2.4}, we deduce that $M^{(i)}$ is abelian. Hence $M$ is solvable, and we are done by \cite[Theorem 3.2]{khanh-hai} or \cite{moghaddam}. 

\bigskip	

\textit{Case 2:} $N_G(R^*)=G$. 

\bigskip

	Since $R$ is a division ring normalized by $G$, by Stuth's Theorem, either $F[M^{(i)}]\subseteq F$ or $R=F[M^{(i)}]=F[M]=D$. The first case implies that $M$ is solvable, and hence $[D:F]<\infty$ by \cite[Theorem 3.2]{khanh-hai} or \cite{moghaddam}. Now, suppose the second case occurs, from which it follows by \cite[Theorem 1.1]{wehrfritz87} that $M$ is abelian-by-locally finite. Let $A$ be an abelian normal subgroup of $M$ such that $M/A$ is locally finite. We have the two following subcases:

\bigskip	

\textit{Subcase 2.1:} $ M^{(i)}\cap A\subseteq F$. 

\bigskip

	We know that $M^{(i)}/M^{(i)}\cap A\cong M^{(i)}A/A\leq M/A$ is a locally finite group. By Lemma \ref{lemma_2.5}, it follows that $D=F[M^{(i)}]$ is a locally finite division ring. Since $M$ is locally solvable, it contains no non-cyclic free subgroups. Thus, by  \cite[Theorem 3.1]{hai-khanh}, we have $[D:F]<\infty$.

\bigskip	

\textit{Subcase 2.2:} $ M^{(i)}\cap A\not\subseteq F$. 

\bigskip

	In this case, there is an element $x\in (M^{(i)}\cap A)\backslash F$ which is algebraic over $F$. By the same arguments used in Case 1 of the proof of Lemma \ref{lemma_2.3}, we have  $[D:F]<\infty$. \\

	Setting $n=[D:F]$, we know that $D\otimes_FD^{op}\cong \M_n(F)$. Thus, by viewing $M$ as a subgroup of $\GL_n(F)$, we conclude that $M$ is a solvable group and the results follow from \cite[Theorem 3.2]{khanh-hai} or \cite{moghaddam}.

\end{proof}


\begin{thebibliography}{}	
	
		\bibitem{moghaddam} R. Fallah-Moghaddam, Maximal subgroups of $\SL_n(D)$, J. Algebra 531 (2019) 70-82.
		
		\bibitem{hai-ha} B. X. Hai, D. V. P. Ha, On locally soluble maximal subgroups of the multiplicative group of a division ring, Vietnam J. Math. 38(2) (2010), 237- 247.
		
		\bibitem{hai-khanh}  B. X. Hai and H. V. Khanh, Free subgroups in maximal subgroups of skew linear groups, Internat. J. Algebra Comput. 29(3) (2019) 603-614.
		
		\bibitem{hai-thin} B. X. Hai, N. V. Thin, On locally nilpotent subgroups of $GL_1(D)$,  Comm. Algebra 37 (2009), 712-718.
		
		\bibitem{hai-tu}  B. X. Hai, N. A. Tu, On multiplicative subgroups in division rings, J. Algebra  Appl. 15(3) 1650050 (2016) (16 pages).
		
		\bibitem{her} I. N. Herstein,  Multiplicative commutators in division rings, Israel J. Math. 31 (1978), 180-188.
		
		\bibitem{huz} M. S. Huzurbazar, The multiplicative group of a division ring, Soviet Math. Dokl. (1960), 1433-1435.
		
		\bibitem{khanh-hai} H. V. Khanh, B. X. Hai, A note on solvable maximal subgroups in subnormal subgroups of $\GL_n(D)$, arXiv:1809.00356v2 [math.RA].
		
		\bibitem{robinson} D. J. S. Robinson, A Course in the Theory of Groups (2nd edn, Springer, 1995).
		
		\bibitem{stuth} C. J. Stuth, A generalization of the Cartan-Brauer-Hua Theorem, Proc. Amer. Math. Soc. 15(2) (1964), 211-217.
	    	    		
		\bibitem{wehrfritz87} B. A. F. Wehrfritz, Locally soluble skew linear groups, Math. Proc. Cambridge Philos. Soc. 102 (1987), 421-429.
		
		\bibitem{wehrfritz_89} B. A. F. Wehrfritz, Goldie subrings of Artinian rings generated by groups, Q. J. Math. Oxford 40 (1989), 501-512.
		
		\bibitem{wehrfritz_91} B. A. F. Wehrfritz, Crossed product criteria and skew linear groups, J. Algebra 141 (1991) 321–353.
\end{thebibliography}
\end{document}